\documentclass[11pt]{amsart}
\usepackage{mathrsfs,latexsym,amsfonts,amssymb}
\newtheorem{theorem}{Theorem}[section]
\newtheorem{lemma}[theorem]{Lemma}
\newtheorem{corollary}[theorem]{Corollary}
\newtheorem{question}[theorem]{Question}
\newtheorem{example}[theorem]{Example}
\theoremstyle{definition}
\newtheorem{definition}[theorem]{Definition}

\theoremstyle{remark}

\begin{document}

\title[The strong Pytkeev property and strong countable completeness in (strongly) topological gyrogroups]
{The strong Pytkeev property and strong countable completeness in (strongly) topological gyrogroups}

\author{Meng Bao}
\address{(Meng Bao): College of Mathematics, Sichuan University, Chengdu 610064, P. R. China}
\email{mengbao95213@163.com}

\author{Xiaoyuan Zhang}
\address{(Xiaoyuan Zhang): 1. College of Mathematics, Sichuan University, Chengdu 610064, P. R. China; 2. School of Big Data Science, Hebei Finance University, Baoding 071051, P. R. China}
\email{405518791@qq.com}

\author{Xiaoquan Xu*}
\address{(Xiaoquan Xu): School of mathematics and statistics,
Minnan Normal University, Zhangzhou 363000, P. R. China}
\email{xiqxu2002@163.com}

\thanks{The authors are supported by the National Natural Science Foundation of China (11661057,12071199) and the Natural Science Foundation of Jiangxi Province, China (20192ACBL20045)\\
*corresponding author}


\begin{abstract}
A topological gyrogroup is a gyrogroup endowed with a topology such that the binary operation is jointly continuous and the inverse mapping is also continuous. In this paper, it is proved that if $G$ is a sequential topological gyrogroup with an $\omega^{\omega}$-base, then $G$ has the strong Pytkeev property. Moreover, some equivalent conditions about $\omega^{\omega}$-base and strong Pytkeev property are given in Baire topological gyrogroups. Finally, it is shown that if $G$ is a strongly countably complete strongly topological gyrogroup, then $G$ contains a closed, countably compact, admissible subgyrogroup $P$ such that the quotient space $G/P$ is metrizable and the canonical homomorphism $\pi :G\rightarrow G/P$ is closed.
\end{abstract}

\maketitle
\section{Introduction}

As we all know, first-countability as an important and basic topological property has been researched for many years. During the times, various topological properties generalizing first-countability have been posed. For example, following \cite{PEG}, Pytkeev claimed that every sequential space satisfies a property which is stronger than countable tightness. Then, in \cite{MT}, Malykhin and Tironi named the property {\it the Pytkeev property}. Furthermore, Tsaban and Zdomskyy \cite{TZ} strengthened this property and posed a concept of the strong Pytkeev property.

The strong Pytkeev property is usually studied combining the other spaces, such as topological groups, topological vector spaces, etc., see \cite{BTL,GKL1,LDRK,LX,SM}. In this paper, we mainly research the strong Pytkeev property in topological gyrogroups. The concept of a gyrogroup was introduced by Ungar in \cite{UA1988,UA} when he researched the $c$-ball of relativistically admissible velocities with the Einstein velocity addition. It is well-known that a gyrogroup has a weaker algebraic structure than a group. Then, Atiponrat \cite{AW} gave the concept of topological gyrogroups, that is, a topological gyrogroup is a gyrogroup endowed with a topology such that the binary operation is jointly continuous and the inverse mapping is also continuous. He proved that $T_{0}$ and $T_{3}$ are equivalent in topological gyrogroups. Moreover, he gave some examples of topological gyrogroups, such as M\"{o}bius gyrogroups, Einstein gyrogroups, and Proper Velocity gyrogroups, that were studied in \cite{FM, FM1,FM2,UA}. After then, Cai, Lin and He in \cite{CZ} proved that every topological gyrogroup is a rectifiable space and deduced that first-countability and metrizability are equivalent in topological gyrogroups. In fact, this kind of space has been studied for many years, see \cite{AW1,AW2020,BL2,BL3,BLX,LF,LF1,LF2,LF3,SL,ST,ST1,ST2,UA2005,UA2002,WAS2020}. In 2019, Bao and Lin \cite{BL} defined the concept of strongly topological gyrogroups and claimed that M\"{o}bius gyrogroups, Einstein gyrogroups, and Proper Velocity gyrogroups are all strongly topological gyrogroups. Furthermore, they gave an example to show that there exists a strongly topological gyrogroup which has an infinite $L$-subgyrogroup.

This paper is organized as follows. In Section 3, we mainly research the strong Pytkeev property in topological gyrogroups. We show that if $G$ is a topological gyrogroup with an $\omega^{\omega}$-base $\{U_{\alpha}:\alpha \in \mathbb{N}^{\mathbb{N}}\}$ and the set $\bigcup _{k\in \mathbb{N}}D_{k}(\alpha)$ is a neighborhood of the identity element $0$ for all $\alpha \in \mathbb{N}^{\mathbb{N}}$, then $G$ has the strong Pytkeev property. Moreover, we claim that if $G$ is a sequential topological gyrogroup with an $\omega^{\omega}$-base $\{U_{\alpha}:\alpha \in \mathbb{N}^{\mathbb{N}}\}$, then the set $\bigcup _{k\in \mathbb{N}}D_{k}(\alpha)$ is a neighborhood of the identity element $0$ for all $\alpha \in \mathbb{N}^{\mathbb{N}}$. The two results above can deduce that if $G$ is a sequential topological gyrogroup with an $\omega^{\omega}$-base $\{U_{\alpha}:\alpha \in \mathbb{N}^{\mathbb{N}}\}$, then $G$ has the strong Pytkeev property. In Section 4, we study the strongly countably complete property in strongly topological gyrogroups. We claim that if $G$ is a strongly countably complete strongly topological gyrogroup, then $G$ contains a closed, countably compact, admissible subgyrogroup $P$ such that the quotient space $G/P$ is metrizable and the canonical homomorphism $\pi :G\rightarrow G/P$ is closed.

\section{Preliminaries}

Throughout this paper, all topological spaces are assumed to be Hausdorff, unless otherwise is explicitly stated. Let $\mathbb{N}$ be the set of all positive integers and $\omega$ the first infinite ordinal. The readers may consult \cite{AA, E, linbook, UA} for notation and terminology not explicitly given here. Next we recall some definitions and facts.

\begin{definition}\cite{UA}
Let $(G, \oplus)$ be a groupoid. The system $(G,\oplus)$ is called a {\it gyrogroup}, if its binary operation satisfies the following conditions:

\smallskip
(G1) There exists a unique identity element $0\in G$ such that $0\oplus a=a=a\oplus0$ for all $a\in G$;

\smallskip
(G2) For each $x\in G$, there exists a unique inverse element $\ominus x\in G$ such that $\ominus x \oplus x=0=x\oplus (\ominus x)$;

\smallskip
(G3) For all $x, y\in G$, there exists $\mbox{gyr}[x, y]\in \mbox{Aut}(G, \oplus)$ with the property that $x\oplus (y\oplus z)=(x\oplus y)\oplus \mbox{gyr}[x, y](z)$ for all $z\in G$, and

\smallskip
(G4) For any $x, y\in G$, $\mbox{gyr}[x\oplus y, y]=\mbox{gyr}[x, y]$.
\end{definition}

Notice that a group is a gyrogroup $(G,\oplus)$ such that $\mbox{gyr}[x,y]$ is the identity function for all $x, y\in G$. The definition of a subgyrogroup is given as follows.

\begin{definition}\cite{ST}
Let $(G,\oplus)$ be a gyrogroup. A nonempty subset $H$ of $G$ is called a {\it subgyrogroup}, denoted
by $H\leq G$, if $H$ forms a gyrogroup under the operation inherited from $G$ and the restriction of $gyr[a,b]$ to $H$ is an automorphism of $H$ for all $a,b\in H$.

\smallskip
Furthermore, a subgyrogroup $H$ of $G$ is said to be an {\it $L$-subgyrogroup}, denoted
by $H\leq_{L} G$, if $gyr[a, h](H)=H$ for all $a\in G$ and $h\in H$.
\end{definition}

\begin{lemma}\cite{UA}\label{a}
Let $(G, \oplus)$ be a gyrogroup. Then for any $x, y, z\in G$, we obtain the following:

\begin{enumerate}
\smallskip
\item $(\ominus x)\oplus (x\oplus y)=y$. \ \ \ (left cancellation law)

\smallskip
\item $(x\oplus (\ominus y))\oplus gyr[x, \ominus y](y)=x$. \ \ \ (right cancellation law)

\smallskip
\item $(x\oplus gyr[x, y](\ominus y))\oplus y=x$.

\smallskip
\item $gyr[x, y](z)=\ominus (x\oplus y)\oplus (x\oplus (y\oplus z))$.
\end{enumerate}
\end{lemma}

\begin{definition}\cite{AW}
A triple $(G, \tau, \oplus)$ is called a {\it topological gyrogroup} if the following statements hold:

\smallskip
(1) $(G, \tau)$ is a topological space.

\smallskip
(2) $(G, \oplus)$ is a gyrogroup.

\smallskip
(3) The binary operation $\oplus: G\times G\rightarrow G$ is jointly continuous while $G\times G$ is endowed with the product topology, and the operation of taking the inverse $\ominus (\cdot): G\rightarrow G$, i.e. $x\rightarrow \ominus x$, is also continuous.
\end{definition}

Obviously, every topological group is a topological gyrogroup. However, every topological gyrogroup whose gyrations are not identically equal to the identity is not a topological group.

\begin{example}\cite{AW}
The Einstein gyrogroup with the standard topology is a topological gyrogroup but not a topological group.
\end{example}

Let $\mathbb{R}_{\mathbf{c}}^{3}=\{\mathbf{v}\in \mathbb{R}^{3}:||\mathbf{v}||<\mathbf{c}\}$, where $\mathbf{c}$ is the vacuum speed of light, and $||\mathbf{v}||$ is the Euclidean norm of a vector $\mathbf{v}\in \mathbb{R}^{3}$. The Einstein velocity addition $\oplus _{E}:\mathbb{R}_{\mathbf{c}}^{3}\times \mathbb{R}_{\mathbf{c}}^{3}\rightarrow \mathbb{R}_{\mathbf{c}}^{3}$ is given as follows: $$\mathbf{u}\oplus _{E}\mathbf{v}=\frac{1}{1+\frac{\mathbf{u}\cdot \mathbf{v}}{\mathbf{c}^{2}}}(\mathbf{u}+\frac{1}{\gamma _{\mathbf{u}}}\mathbf{v}+\frac{1}{\mathbf{c}^{2}}\frac{\gamma _{\mathbf{u}}}{1+\gamma _{\mathbf{u}}}(\mathbf{u}\cdot \mathbf{v})\mathbf{u}),$$ for any $\mathbf{u,v}\in \mathbb{R}_{c}^{3}$, $\mathbf{u}\cdot \mathbf{v}$ is the usual dot product of vectors in $\mathbb{R}^{3}$, and $\gamma _{\mathbf{u}}$ is the gamma factor which is given by $$\gamma _{\mathbf{u}}=\frac{1}{\sqrt{1-\frac{\mathbf{u}\cdot \mathbf{u}}{\mathbf{c}^{2}}}}.$$

It was proved in \cite{UA} that $(\mathbb{R}^{3}_{c},\oplus _{E})$ is a gyrogroup but not a group. Moreover, with the standard topology inherited from $\mathbb{R}^{3}$, it is clear that $\oplus _{E}$ is continuous. Finally, $-\mathbf{u}$ is the inverse of $\mathbf{u}\in \mathbb{R}^{3}$ and the operation of taking the inverse is also continuous. Therefore, the Einstein gyrogroup $(\mathbb{R}^{3}_{c},\oplus _{E})$ with the standard topology inherited from $\mathbb{R}^{3}$ is a topological gyrogroup but not a topological group.

\begin{definition}\cite{BT,LPT,GK}
A point $x$ of a topological space $X$ is said to have a {\it neighborhood $\omega^{\omega}$-base} or a {\it local $\mathfrak{G}$-base} if there exists a base of neighborhoods at $x$ of the form $\{U_{\alpha}(x):\alpha \in \mathbb{N}^{\mathbb{N}}\}$ such that $U_{\beta}(x)\subset U_{\alpha}(x)$ for all elements $\alpha \leq \beta$ in $\mathbb{N}^{\mathbb{N}}$, where $\mathbb{N}^{\mathbb{N}}$ consisting of all functions from $\mathbb{N}$ to $\mathbb{N}$ is endowed with the natural partial order, ie., $f\leq g$ if and only if $f(n)\leq g(n)$ for all $n\in \mathbb{N}$. The space $X$ is said to have an {\it $\omega^{\omega}$-base} or a {\it $\mathfrak{G}$-base} if it has a neighborhood $\omega^{\omega}$-base or a local $\mathfrak{G}$-base at every point $x\in X$.
\end{definition}

Then we define the concept of an {\it $\omega^{\omega}$-base} or a {\it $\mathfrak{G}$-base} in topological gyrogroups.

\begin{definition}
Let $G$ be a topological gyrogroup. A family $\mathcal{U}=\{U_{\alpha}:\alpha \in \mathbb{N}^{\mathbb{N}}\}$ of neighborhoods of the identity element $0$ is called an {\it $\omega^{\omega}$-base} or a {\it $\mathfrak{G}$-base} if $\mathcal{U}$ is a base of neighborhoods at $0$ and $U_{\beta}\subset U_{\alpha}$ whenever $\alpha \leq \beta$ for all $\alpha,\beta \in \mathbb{N}^{\mathbb{N}}$.
\end{definition}

A topological space $Y$ has the {\it strong Pytkeev property} \cite{TZ} if for each $y\in Y$, there exists a countable family $\mathcal{D}$ of subsets of $Y$, such that for each neighborhood $U$ of $y$ and each $A\subset Y$ with $y\in \overline{A}\setminus A$, there is $D\in \mathcal{D}$ such that $D\subset U$ and $D\cap A$ is infinite.

Then we define this property for topological gyrogroups.

\begin{definition}
A topological gyrogroup $G$ has the {\it strong Pytkeev property} if there exists a sequence $\mathcal{D}=\{D_{n}\}_{n\in \mathbb{N}}$ of subsets of $G$ such that for each neighborhood $U$ of the identity $0$ and each $A\subset G$ with $0\in \overline{A}\setminus A$, there is $n\in \mathbb{N}$ such that $D_{n}\subset U$ and $D_{n}\cap A$ is infinite.
\end{definition}

\begin{definition}\cite{GKS}
A family $\mathcal{N}$ of subsets of a topological space $X$ is called a {\it $cn$-network} at a point $x\in X$ if for each neighborhood $O_{x}$ of $x$ the set $\bigcup \{N\in \mathcal{N}:x\in N\subset O_{x}\}$ is a neighborhood of $x$; $\mathcal{N}$ is a $cn$-network in $X$ if $\mathcal{N}$ is a $cn$-network at each point $x\in X$.

A family $\mathcal{N}$ of subsets of a topological space $X$ is called a {\it $ck$-network} at a point $x\in X$ if for each compact subset $K\subset O_{x}$ there exists a finite subfamily $\mathcal{F}\subset \mathcal{N}$ satisfying $x\in \bigcap \mathcal{F}$ and $K\subset \bigcup \mathcal{F}\subset O_{x}$; $\mathcal{N}$ is a $ck$-network in $X$ if $\mathcal{N}$ is a $ck$-network at each point $x\in X$.
\end{definition}

\section{Topological gyrogroups with strong Pytkeev property}
In this Section, we mainly research topological gyrogroups with $\omega^{\omega}$-base and strong Pytkeev property. We show that if $G$ is a topological gyrogroup with an $\omega^{\omega}$-base $\{U_{\alpha}:\alpha \in \mathbb{N}^{\mathbb{N}}\}$ and the set $\bigcup _{k\in \mathbb{N}}D_{k}(\alpha)$ is a neighborhood of the identity element $0$ for all $\alpha \in \mathbb{N}^{\mathbb{N}}$, then $G$ has the strong Pytkeev property. Moreover, we claim that if $G$ is a sequential topological gyrogroup with an $\omega^{\omega}$-base $\{U_{\alpha}:\alpha \in \mathbb{N}^{\mathbb{N}}\}$, then the set $\bigcup _{k\in \mathbb{N}}D_{k}(\alpha)$ is a neighborhood of the identity element $0$ for all $\alpha \in \mathbb{N}^{\mathbb{N}}$. Therefore, we conclude that if $G$ is a sequential topological gyrogroup with an $\omega^{\omega}$-base $\{U_{\alpha}:\alpha \in \mathbb{N}^{\mathbb{N}}\}$, then $G$ has the strong Pytkeev property. Finally, we give some equivalent conditions about $\omega^{\omega}$-base and strong Pytkeev property in Baire topological gyrogroups.

\bigskip
For every $\alpha =(\alpha_{i})_{i\in \mathbb{N}}\in \mathbb{N}^{\mathbb{N}}$ and each $k\in \mathbb{N}$, set $$I_{k}(\alpha)=\{\beta \in \mathbb{N}^{\mathbb{N}}:\beta _{i}=\alpha _{i} ~for~i=1,...,k\}.$$ Indeed, $I_{k}(\alpha)$ is defined by the finite subset $\{\alpha_{1},...,\alpha_{k}\}$ of $\mathbb{N}$. Therefore, the family $\{I_{k}(\alpha):k\in \mathbb{N},\alpha \in \mathbb{N}^{\mathbb{N}}\}$ is countable. Moreover, suppose that a topological gyrogroup $G$ has an $\omega^{\omega}$-base $\{U_{\alpha}:\alpha \in \mathbb{N}^{\mathbb{N}}\}$. Set $$D_{k}(\alpha)=\bigcap_{\beta \in I_{k}(\alpha)}U_{\beta}, ~~where~~\alpha =(\alpha _{i})_{i\in \mathbb{N}}\in \mathbb{N}^{\mathbb{N}}~~ and ~~k\in \mathbb{N}.$$ It is clear that $D_{k}(\alpha)\subset U_{\alpha}$ and $D_{k}(\alpha)\subset D_{m}(\alpha)$ for every $\alpha \in \mathbb{N}^{\mathbb{N}},k\in \mathbb{N}$ and every natural number $k\leq m$. Set $D_{0}(\alpha)=\{0\}$, for every $\alpha =(\alpha_{i})_{i\in \mathbb{N}}\in \mathbb{N}^{\mathbb{N}}.$ Then, put $\mathcal{D}=\{D_{k}(\alpha):k\in \mathbb{N},\alpha \in \mathbb{N}^{\mathbb{N}}\}.$

For every $\alpha =(\alpha_{i})_{i\in \mathbb{N}}\in \mathbb{N}^{\mathbb{N}}$ and each $k\in \mathbb{N}$, set $K_{\alpha}=\Pi _{i\in \mathbb{N}}[1,\alpha_{i}]\subset \mathbb{N}^{\mathbb{N}}$, $$L_{0}(\alpha)=\mathbb{N}^{\mathbb{N}}~~and~~L_{k}(\alpha)=\bigcup_{\beta \in I_{k}(\alpha)}K_{\beta}=\Pi_{i=1}^{k}[1,\alpha_{i}]\times \mathbb{N}^{\mathbb{N}\setminus \{1,...,k\}}.$$

\begin{lemma}\cite{GKL1}\label{3yl1}
Let $\alpha =(\alpha_{i})_{i\in \mathbb{N}}\in \mathbb{N}^{\mathbb{N}}$ and $\beta^{jk}=(\beta_{i}^{jk})_{i\in \mathbb{N}}\in L_{k-1}(\alpha)\setminus L_{k}(\alpha)$ for every $k\in \mathbb{N}$ and each $1\leq j\leq s_{k}<\infty$. Then there is $\gamma \in \mathbb{N}^{\mathbb{N}}$ such that $\alpha \leq \gamma$ and $\beta^{jk}\leq \gamma$ for every $k\in \mathbb{N}$ and each $1\leq j\leq s_{k}$.
\end{lemma}

\begin{theorem}\label{2dl5}
Suppose that $G$ is a topological gyrogroup with an $\omega^{\omega}$-base $\{U_{\alpha}:\alpha \in \mathbb{N}^{\mathbb{N}}\}$. Suppose further that the set $\bigcup _{k\in \mathbb{N}}D_{k}(\alpha)$ is a neighborhood of the identity element $0$ for all $\alpha \in \mathbb{N}^{\mathbb{N}}$. Then $G$ has the strong Pytkeev property.
\end{theorem}

\begin{proof}
Let $A\subset G$ be such that $0\in \overline{A}\setminus A$. So, for every $\alpha \in \mathbb{N}^{\mathbb{N}}$, the set $A\cap U_{\alpha}$ is infinite. Since $W=\bigcup_{k\in \mathbb{N}}D_{k}(\alpha)$ is a neighborhood of the identity element $0$, the intersection $A\cap W=\bigcup_{k\in \mathbb{N}}(A\cap [D_{k}(\alpha)\setminus D_{k-1}(\alpha)])$ is infinite. For an arbitrary neighborhood $U$ of $0$ in $G$, there exists $\alpha \in \mathbb{N}^{\mathbb{N}}$ such that $U_{\alpha}\subset U$. Then $D_{k}(\alpha)\subset U_{\alpha}\subset U$. Put $A_{k}=A\cap [D_{k}(\alpha)\setminus D_{k-1}(\alpha)]$ for all $k\in \mathbb{N}$. It suffices to show that $A_{k}$ is infinite for some $k\in \mathbb{N}$.

{\bf Claim} There exists $k\in \mathbb{N}$ such that $A_{k}$ is infinite.

Suppose on the contrary, for every $k\in \mathbb{N}$, $A$ is finite. Then we can find an infinite subset $I$ of $\mathbb{N}$ such that $A_{k}=\{a_{1}^{k},...,a_{s_{k}}^{k}\}$ if for all $k\in I$ for some natural number $s_{k}$ and $A_{k}=\emptyset$ if $k\not \in I$.

For all $k\in I$, take $\beta^{jk}=(\beta_{i}^{jk})_{i\in \mathbb{N}}\in I_{k-1}(\alpha)$ such that $a_{j}^{k}\not \in U_{\beta^{jk}}$ for every $1\leq j\leq s_{k}$. By Lemma \ref{3yl1}, there is a $\gamma \in \mathbb{N}^{\mathbb{N}}$ such that $\alpha \leq \gamma$ and $\beta^{jk}\leq \gamma$ for every $k\in I$ and each $1\leq j\leq s_{k}$. Therefore, for all $k\in I$ and every $1\leq j\leq s_{k}$, $a_{j}^{k}\not \in U_{\gamma}$. Since $W$ is a neighborhood of $0$, we can find $\delta \in \mathbb{N}^{\mathbb{N}}$, $\gamma \leq \delta$, such that $U_{\delta}\subset W$. Then $A\cap U_{\delta}$ is empty and thus $0\not \in \overline{A}\setminus A$, which is a contradiction.

\end{proof}

Naturally, we have the following question.

\begin{question}
Suppose that a topological gyrogroup $G$ with an $\omega^{\omega}$-base $\mathcal{U}=\{U_{\alpha}: \alpha \in \mathbb{N}^{\mathbb{N}}\}$ has the strong Pytkeev property. Is the set $\bigcup_{k\in \mathbb{N}}D_{k}(\alpha)$ a neighborhood of the identity element $0$ of $G$ for all $\alpha \in \omega^{\omega}$?
\end{question}

Then we show that if $G$ is a sequential topological gyrogroup with an $\omega^{\omega}$-base $\mathcal{U}=\{U_{\alpha}:\alpha \in \mathbb{N}^{\mathbb{N}}\}$, then $G$ has the strong Pytkeev property.

\begin{lemma}\label{2yl2}
Let $G$ be a topological gyrogroup with an $\omega^{\omega}$-base $\mathcal{U}=\{U_{\alpha}:\alpha \in \mathbb{N}^{\mathbb{N}}\}$. If $G$ is sequential, then the set $\bigcup_{k\in \mathbb{N}}D_{k}(\alpha)$ is an open neighborhood of the identity element $0$ for any $\alpha \in \mathbb{N}^{\mathbb{N}}$.
\end{lemma}

\begin{proof}
Let $A=\bigcup_{k\in \mathbb{N}}D_{k}(\alpha)$. Since $G$ is sequential, it suffices to prove that $A$ is sequentially open. Suppose that $x\in A$ and $\{x_{n}\}_{n\in \mathbb{N}}$ is a sequence converging to $x$ in $G$.

{\bf Claim.} There exists $N\in \mathbb{N}$ such that $x_{n}\in A$ for every $n>N$.

Suppose on the contrary, let $m$ be the minimal index such that $x\in D_{m}(\alpha)$. For every $\beta \in I_{m}(\alpha)$, $x\in U_{\beta}$. Since there exists $n_{1}$ such that $x_{n_{1}}\not \in A$, $x_{n_{1}}\not \in D_{m+1}(\alpha)$. Hence, $x_{n_{1}}\not \in U_{\beta^{1}}$ for some $\beta^{1}\in I_{m}(\alpha)$. There exists $n_{2}>n_{1}$ such that $x_{n_{2}}\not \in A$. Then, $x_{n_{2}}\not \in D_{m+2}(\alpha)$. For some $\beta^{2}\in I_{m+2}(\alpha)$, $x_{n_{2}}\not \in U_{\beta^{2}}$. By induction, we obtain a subsequence $\{x_{n_{k}}\}_{k}$ of $\{x_{n}\}_{n}$ and a sequence $\{\beta^{k}\}_{k}$ in $\mathbb{N}^{\mathbb{N}}$ such that $$x_{n_{k}}\not \in U_{\beta^{k}}~~and~~\beta^{k}\in I_{m+k}(\alpha)~~for~~every~~k\in \mathbb{N}.$$

Let $\gamma =(\gamma_{i})_{i\in \mathbb{N}}$, where $\gamma_{i}=\alpha_{i}$ if $1\leq i\leq m$ and $\gamma_{i}=max\{\beta_{i}^{1},\beta_{i}^{2},...,\beta_{i}^{i-m}\}$ if $i>m$. Then $x\in D_{m}(\alpha)\subset U_{\gamma}$. Moreover, since $U_{\gamma}\subset U_{\beta^{k}}$, we have $x_{n_{k}}\not \in U_{\gamma}$ for every $k\in \mathbb{N}$. Then $U_{\gamma}$ is open and $x\in U_{\gamma}$. Therefore, $x_{n}\not \rightarrow x$ which is a contradiction.

We conclude that $A$ is sequentially open and hence $A$ is open.
\end{proof}

\begin{theorem}
Let $G$ be a topological gyrogroup with an $\omega^{\omega}$-base $\mathcal{U}=\{U_{\alpha}:\alpha \in \mathbb{N}^{\mathbb{N}}\}$. If $G$ is a sequential space, $G$ has the strong Pytkeev property.
\end{theorem}

\begin{proof}
Since $G$ is a sequential space, it follows from Lemma \ref{2yl2} that the set $\bigcup_{k\in \mathbb{N}}D_{k}(\alpha)$ is an open neighborhood of the identity element $0$ for any $\alpha \in \mathbb{N}^{\mathbb{N}}$. By Theorem \ref{2dl5}, $G$ has the strong Pytkeev property.
\end{proof}

Let $\Omega$ be a set and $I$ be a partially ordered set with an order $\leq$. We say that a family $\{A_{i}\}_{i\in I}$ of subsets of $\Omega$ is $I$-decreasing if $A_{j}\subset A_{i}$ for every $i\leq j$ in $I$. Let $\mathbf{M}\subset \mathbb{N}^{\mathbb{N}}$ and let $\mathcal{U}=\{U_{\alpha}:\alpha \in \mathbf{M}\}$ be an $\mathbf{M}$-decreasing family of subsets of a set $\Omega$. Let $$DM_{k}(\alpha)=\bigcap_{\beta \in I_{k}(\alpha)\cap \mathbf{M}}U_{\beta}.$$ Define a countable family $\mathcal{D}_{\mathcal{U}}=\{DM_{k}(\alpha):\alpha \in \mathbf{M},k\in \mathbb{N}\}$, where $\mathcal{U}$ satisfies the condition $(\mathbf{D})$ if $U_{\alpha}=\bigcup_{k\in \mathbb{N}}DM_{k}(\alpha)$ for every $\alpha \in \mathbf{M}$, see \cite{GK}.

\begin{lemma}\label{2yl1}
If $G$ is a topological gyrogroup and has the strong Pytkeev property with a sequence $\mathcal{D}=\{D_{n}\}_{n\in \mathbb{N}}$, for every neighborhood $U$ of the identity element $0$, there is $\alpha =(\alpha_{i})_{i\in \mathbb{N}}\in \mathbb{N}^{\mathbb{N}}$ such that the set $N_{\alpha}=\bigcup_{i\in \mathbb{N}}(D_{\alpha_{i}}\cup \{0\})$ is a neighborhood of $0$ and $N_{\alpha}\subset U$.
\end{lemma}

\begin{proof}
Set $J=\{n\in N:D_{n}\subset U\}$. Then $J=\{\alpha_{i}\}_{i\in \mathbb{N}}$, where $\alpha_{1}<\alpha_{2}<\cdot \cdot \cdot$. Set $\alpha =(\alpha_{i})_{i\in \mathbb{N}}$. Therefore, $\alpha \in \mathbb{N}^{\mathbb{N}}$ and $N_{\alpha}\subset U$.

Suppose on the contrary that $N_{\alpha}$ is not a neighborhood of the identity element $0$. Then $0\in \overline{U\setminus N_{\alpha}}$. Since $0\not \in U\setminus N_{\alpha}$, it follows that $0\in (\overline{U\setminus N_{\alpha}})\setminus (U\setminus N_{\alpha})$. By the definition of the Pytkeev property, there exists $m\in \mathbb{N}$ such that $D_{m}\subset U$ and $D_{m}\cap (U\setminus N_{\alpha})$ is infinite. Therefore, it is a contradiction with the choice of $J$ and the definition of $N_{\alpha}$. Hence, $N_{\alpha}$ is a neighborhood of $0$.
\end{proof}

\begin{theorem}
If $G$ is a topological gyrogroup and has the strong Pytkeev property with a sequence $\mathcal{D}=\{D_{n}\}_{n\in \mathbb{N}}$, then $G$ has a base $\{U_{\alpha}:\alpha \in \mathbf{M}\}$ of neighborhood at $0$, where

(1) $\mathbf{M}$ is a subset of the partially ordered set $\mathbb{N}^{\mathbb{N}}$;

\smallskip
(2) If $\alpha \in \mathbf{M}$ and $\beta \in \mathbb{N}^{\mathbb{N}}$ are such that $\beta \leq \alpha$, then $\beta \in \mathbf{M}$, and

(3) $U_{\beta}\subset U_{\alpha}$, where $\alpha \leq \beta$ for $\alpha ,\beta \in \mathbf{M}$.
\end{theorem}

\begin{proof}
Since every topological gyrogroup is homogeneous, without loss of generality, we suppose that $0\in D_{n}$ for every $n\in \mathbb{N}$. For each $k,i\in \mathbb{N}$, set $D_{k}^{i}=\bigcap_{l=1}^{k}D_{i-1+l}$. So the sequence $\{D_{k}^{i}\}_{k\in \mathbb{N}}$ is decreasing for every $i\in \mathbb{N}$. Furthermore, for each $\alpha =(\alpha_{i})_{i\in \mathbb{N}}\in \mathbb{N}^{\mathbb{N}}$, put $U_{\alpha}=\bigcup_{i\in \mathbb{N}}D_{\alpha_{i}}^{i}$. It is clear that $U_{\alpha}\subset U_{\beta}$ for each $\alpha ,\beta \in \mathbb{N}^{\mathbb{N}}$ with $\beta \leq \alpha$.

Fix an increasing sequence $0=n_{0}<n_{1}<n_{2}<\cdot \cdot \cdot$ in $\mathbb{N}$ such that $\bigcup_{k\in \mathbb{N}}D_{n_{k}}$ is a neighborhood at $0$.

{\bf Claim.} There is $\alpha =(\alpha_{i})_{i\in \mathbb{N}}\in \mathbb{N}^{\mathbb{N}}$ such that $U_{\alpha}=\bigcup_{k\in \mathbb{N}}D_{n_{k}}$.

Indeed, if $i=n_{k}$ for some $k\in \mathbb{N}$, we set $\alpha_{i}=1$. So $D_{\alpha_{i}}^{i}=D_{n_{k}}$. However, if $n_{k-1}<i<n_{k}$ for some $k\in \mathbb{N}$, we set $\alpha_{i}=n_{k}-i+1$. Then $D_{\alpha_{i}}^{i}=\bigcap_{l=1}^{\alpha_{i}}D_{i-1+l}\subset D_{n_{k}}$. Thus, $U_{\alpha}=\bigcup_{k\in \mathbb{N}}D_{n_{k}}$.

Set $M=\{\alpha \in \mathbb{N}^{\mathbb{N}}:U_{\alpha}~~is~~a~~neighborhood~~of~~0\}$. By Lemma \ref{2yl1}, the set $\{U_{\alpha}:\alpha \in \mathbf{M}\}$ forms a base at $0$ satisfying (iii). Indeed, let $\alpha \in \mathbf{M}$ and $\beta \in \mathbb{N}^{\mathbb{N}}$ be such that $\beta \leq \alpha$, then $U_{\alpha}\subset U_{\beta}$. Therefore, $U_{\beta}$ is a neighborhood of the identity element $0$. Hence, $\beta \in \mathbf{M}$.
\end{proof}

\begin{corollary}
If a topological gyrogroup $G$ has the strong Pytkeev property, then $\chi (G) \leq \mathfrak{c}$.
\end{corollary}

Finally in this section, we give some equivalent conditions about Baire topological gyrogroups.

\begin{lemma}\cite{GK}\label{1yl2}
Let $x$ be a point of a topological space $X$. Then $X$ has a countable $cn$-network at $x$ if and only if $X$ has a small base $\mathcal{U}(x)=\{U_{\alpha}:\alpha \in \mathbf{M}_{x}\}$ at $x$ satisfying the condition $(\mathbf{D})$. In that case the family $\mathcal{D}_{\mathcal{U}(x)}$ is a countable $cn$-network at $x$.
\end{lemma}

\begin{theorem}
Let $G$ be a Baire topological gyrogroup. Then the following are equivalent:

(i) $G$ is metrizable.

\smallskip
(ii) $G$ has the strong Pytkeev property.

\smallskip
(iii) $G$ has countable $ck$-character.

\smallskip
(iv) $G$ has countable $cn$-character.

\smallskip
(v) $G$ has an $\omega^{\omega}$-base satisfying the condition $(\mathbf{D})$.
\end{theorem}

\begin{proof}
The implications (i) $\Rightarrow$ (ii) and (iii) $\Rightarrow$ (iv) are trivial. Moreover, (ii) $\Rightarrow$ (iii) directly follows from the fact that every countable $cp$-network at a point $x$ of a topological space $X$ is a $ck$-network at $x$ \cite{BTL}. Then (v) $\Rightarrow$ (iv) follows from Lemma \ref{1yl2}. Then we show that (i) $\Rightarrow$ (v) and (iv) $\Rightarrow$ (i).

(i) $\Rightarrow$ (v) If $\{V_{n}\}_{n\in \mathbb{N}}$ is a decreasing base of neighborhoods at the identity element $0$ of $G$, the family $\{U_{\alpha}:\alpha \in \mathbb{N}^{\mathbb{N}}\}$, where $U_{\alpha}=V_{\alpha_{1}}$ for $\alpha =(\alpha_{i})\in \mathbb{N}^{\mathbb{N}}$, is an $\omega^{\omega}$-base and satisfies the condition ($\mathbf{D}$).

(iv) $\Rightarrow$ (i) Suppose that $G$ has a countable $cn$-character and we claim that $G$ is first-countable. It follows from Lemma \ref{1yl2} that there is a small local base $\mathcal{U}=\{U_{\alpha}:\alpha \in \mathbf{M}\}$ at $0$ which satisfies the condition $(\mathbf{D})$. For an arbitrary open neighborhood $W$ of $0$, choose a symmetric open neighborhood $V$ of $0$ such that $V\oplus V\subset \overline{V}\oplus \overline{V}\subset W$. Then, we can find $\alpha \in \mathbf{M}$ with $U_{\alpha}=\bigcup_{k}DM_{k}(\alpha)\subset V$ and $Int(U_{\alpha})$ is open in $G$. Moreover, it follows from the Baire property of $G$ that there is $k\in \mathbb{N}$ such that $Int(U_{\alpha})\cap \overline{DM_{k}(\alpha)}$ has a non-empty interior in $U_{\alpha}$. Therefore, $Int(U_{\alpha})\cap \overline{DM_{k}(\alpha)}$ has a non-empty interior in $G$ and hence $\overline{DM_{k}(\alpha)}\oplus \overline{(\ominus DM_{k}(\alpha))}$  is a countable neighborhood of the identity element $0$ which is contained in $W$. Furthermore, since $G$ is homogeneous, $G$ is first-countable. Then, every first-countable topological gyrogroup is metrizable, so $G$ is metrizable.
\end{proof}

\section{Strongly topological gyrogroups with strong countable completeness}
In this Section, we claim that if $G$ is a strongly countably complete strongly topological gyrogroup, then $G$ contains a closed, countably compact, admissible subgyrogroup $P$ such that the quotient space $G/P$ is metrizable and the canonical homomorphism $\pi :G\rightarrow G/P$ is closed.

A space $X$ is called {\it strongly countably complete} \cite{FZ} if there exists a sequence $\{\gamma_{n}:n\in \mathbb{N}\}$ of open covering of $X$ such that every decreasing sequence $\{F_{n}:n\in \mathbb{N}\}$ of nonempty closed sets in $X$ has nonempty intersection provided each $F_{n}$ is contained in some element of $\gamma_{n}$.

\begin{definition}\cite{BL}\label{d11}
Let $G$ be a topological gyrogroup. We say that $G$ is a {\it strongly topological gyrogroup} if there exists a neighborhood base $\mathscr U$ of $0$ such that, for every $U\in \mathscr U$, $\mbox{gyr}[x, y](U)=U$ for any $x, y\in G$. For convenience, we say that $G$ is a strongly topological gyrogroup with neighborhood base $\mathscr U$ of $0$.
\end{definition}

Clearly, we may assume that $U$ is symmetric for each $U\in\mathscr U$ in Definition~\ref{d11}. Moreover, in the
classical M\"{o}bius, Einstein or Proper Velocity gyrogroups, we know that gyrations are indeed special rotations. However, for an arbitrary gyrogroup, gyrations belong to the automorphism group of $G$ and need not be necessarily rotations.

In \cite{BL}, the authors proved that there is a strongly topological gyrogroup which is not a topological group, see Example \ref{lz1}.

\begin{example}\cite{BL}\label{lz1}
Let $\mathbb{D}$ be the complex open unit disk $\{z\in \mathbb{C}:|z|<1\}$. We consider $\mathbb{D}$ with the standard topology. In \cite[Example 2]{AW}, define a M\"{o}bius addition $\oplus _{M}: \mathbb{D}\times \mathbb{D}\rightarrow \mathbb{D}$ to be a function such that $$a\oplus _{M}b=\frac{a+b}{1+\bar{a}b}\ \mbox{for all}\ a, b\in \mathbb{D}.$$ Then $(\mathbb{D}, \oplus _{M})$ is a gyrogroup, and it follows from \cite[Example 2]{AW} that $$gyr[a, b](c)=\frac{1+a\bar{b}}{1+\bar{a}b}c\ \mbox{for any}\ a, b, c\in \mathbb{D}.$$ For any $n\in \mathbb{N}$, let $U_{n}=\{x\in \mathbb{D}: |x|\leq \frac{1}{n}\}$. Then, $\mathscr U=\{U_{n}: n\in \mathbb{N}\}$ is a neighborhood base of $0$. Moreover, we observe that $|\frac{1+a\bar{b}}{1+\bar{a}b}|=1$. Therefore, we obtain that $gyr[x, y](U)\subset U$, for any $x, y\in \mathbb{D}$ and each $U\in \mathscr U$, then it follows that $gyr[x, y](U)=U$ by \cite[Proposition 2.6]{ST}. Hence, $(\mathbb{D}, \oplus _{M})$ is a strongly topological gyrogroup. However, $(\mathbb{D}, \oplus _{M})$ is not a group \cite[Example 2]{AW}.
\end{example}

\bigskip
{\bf Remark.} Even though M\"{o}bius gyrogroups, Einstein gyrogroups, and Proper Velocity gyrogroups are all strongly topological gyrogroups, all of them do not possess any non-trivial $L$-subgyrogroups. However, there is a class of strongly topological gyrogroups which has a non-trivial $L$-subgyrogroup, see Example~\ref{e1}.

\begin{example}\label{e1}\cite{BL}
There exists a strongly topological gyrogroup which has an infinite $L$-subgyrogroup.
\end{example}

\smallskip
Indeed, let $X$ be an arbitrary feathered non-metrizable topological group, and let $Y$ be an any strongly topological gyrogroup with a non-trivial $L$-subgyrogroup (such as the gyrogroup $K_{16}$ \cite[p. 41]{UA2002}). Put $G=X\times Y$ with the product topology and the operation with coordinate. Then $G$ is an infinite strongly topological gyrogroup since $X$ is infinite. Let $H$ be a non-trivial $L$-subgyrogroup of $Y$, and take an arbitrary infinite subgroup $N$ of $X$. Then $N\times H$ is an infinite $L$-subgyrogroup of $G$.

\bigskip
Then, we recall the following concept of the coset space of a topological gyrogroup.

Let $(G, \tau, \oplus)$ be a topological gyrogroup and $H$ an $L$-subgyrogroup of $G$. It follows from \cite[Theorem 20]{ST} that $G/H=\{a\oplus H:a\in G\}$ is a partition of $G$. We denote by $\pi$ the mapping $a\mapsto a\oplus H$ from $G$ onto $G/H$. Clearly, for each $a\in G$, we have $\pi^{-1}\{\pi(a)\}=a\oplus H$.
Denote by $\tau (G)$ the topology of $G$. In the set $G/H$, we define a family $\tau (G/H)$ of subsets as follows: $$\tau (G/H)=\{O\subset G/H: \pi^{-1}(O)\in \tau (G)\}.$$

The following concept of an admissible subgyrogroup of a strongly topological gyrogroup was first introduced in \cite{BL1}.

A subgyrogroup $H$ of a topological gyrogroup $G$ is called {\it admissible} if there exists a sequence $\{U_{n}:n\in \omega\}$ of open symmetric neighborhoods of the identity $0$ in $G$ such that $U_{n+1}\oplus (U_{n+1}\oplus U_{n+1})\subset U_{n}$ for each $n\in \omega$ and $H=\bigcap _{n\in \omega}U_{n}$. If $G$ is a strongly topological gyrogroup with a symmetric neighborhood base $\mathscr U$ at $0$ and each $U_{n}\in \mathscr U$, we say that the admissible topological subgyrogroup is generated from $\mathscr U$.

\begin{lemma}\label{1yl1}\cite{BL2}
Suppose that $(G, \tau, \oplus)$ is a strongly topological gyrogroup with a symmetric neighborhood base $\mathscr U$ at $0$. Then each admissible topological subgyrogroup $H$ generated from $\mathscr U$ is a closed $L$-subgyrogroup of $G$.
\end{lemma}

\begin{lemma}\label{s}\cite{BL}
Let $G$ be a strongly topological gyrogroup with the symmetric neighborhood base $\mathscr{U}$ at $0$, and let $\{U_{n}: n\in\omega\}$ and $\{V(m/2^{n}): n, m\in\omega\}$ be two sequences of open neighborhoods satisfying the following conditions (1)-(5):

\smallskip
(1) $U_{n}\in\mathscr{U}$ for each $n\in \omega$.

\smallskip
(2) $U_{n+1}\oplus U_{n+1}\subset U_{n}$, for each $n\in \omega$.

\smallskip
(3) $V(1)=U_{0}$;

\smallskip
(4) For any $n\geq 1$, put $$V(1/2^{n})=U_{n}, V(2m/2^{n})=V(m/2^{n-1})$$ for $m=1,...,2^{n-1}$, and $$V((2m+1)/2^{n})=U_{n}\oplus V(m/2^{n-1})=V(1/2^{n})\oplus V(m/2^{n-1})$$ for each $m=1,...,2^{n-1}-1$;

\smallskip
(5) $V(m/2^{n})=G$ when $m>2^{n}$;

\smallskip
Then there exists a prenorm $N$ on $G$ satisfies the following conditions:

\smallskip
(a) for any fixed $x, y\in G$, we have $N(\mbox{gyr}[x,y](z))=N(z)$ for any $z\in G$;

\smallskip
(b) for any $n\in\omega$, $$\{x\in G: N(x)<1/2^{n}\}\subset U_{n}\subset\{x\in G: N(x)\leq 2/2^{n}\}.$$
\end{lemma}

\begin{theorem}\label{3dl1}
Let $G$ be a strongly countably complete strongly topological gyrogroup with a symmetric neighborhood base $\mathscr U$. Then $G$ contains a closed, countably compact, admissible subgyrogroup $P$ such that the quotient space $G/P$ is metrizable and the canonical homomorphism $\pi :G\rightarrow G/P$ is closed.
\end{theorem}

\begin{proof}
Let $A$ be a $G_{\delta}$-set in $G$ containing the identity element $0$. Take a family $\lambda =\{W_{n}:n\in \mathbb{N}\}$ of open sets in $G$ such that $A=\bigcap \lambda$. Suppose that $\{\gamma_{n}:n\in \mathbb{N}\}$ is a family of open coverings of $G$ witnessing the strongly countable completeness of $G$. For each $n\in \mathbb{N}$, choose an element $U_{n}\in \gamma_{n}$ containing the identity element $0$ of $G$. Define a sequence $\{V_{n}:n\in \mathbb{N}\}\subset \mathscr U$ by induction such that $V_{0}\subset U_{0}\cap W_{0}$ and $V_{n+1}\oplus (V_{n+1}\oplus V_{n+1})\subset U_{n+1}\cap V_{n}\cap W_{n+1}$ for each $n\in \mathbb{N}$. Put $P=\bigcap_{n\in \mathbb{N}}V_{n}$. By Lemma \ref{1yl1}, $P$ is a closed admissible $L$-subgyrogroup of $G$ and $P\subset \bigcap_{n\in \mathbb{N}}W_{n}=A$.

{\bf Claim 1.} If $x_{n}\in V_{n}$ for each $n\in \mathbb{N}$ and the set $X=\{x_{n}:n\in \mathbb{N}\}$ is infinite, then $X$ has an accumulation point in $P$.

Let $X=\{x_{n}:n\in \mathbb{N}\}$. Since $\overline{V_{n+1}}\subset V_{n+1}\oplus V_{n+1}\subset V_{n}$, the definition of $P$ and the inclusion $x_{n}\in V_{n}$ for each $n\in \mathbb{N}$ together imply that all accumulation points of $X$ lie in $P$. Hence, if $X$ has no accumulation points in $P$, $X$ will be closed and discrete in $G$. Set $F_{n}=\{x_{k}:k\geq n\}$ for each $n\in \mathbb{N}$. The sets $F_{n}$ are closed in $G$ and $F_{n}\subset V_{n}\subset U_{n}\in \gamma_{n}$ for each $n\in \mathbb{N}$. However, $\bigcap_{n\in \mathbb{N}}F_{n}=\emptyset$, which is a contradiction with the choice of the family $\{\gamma_{n}:n\in \mathbb{N}\}$.

{\bf Claim 2.} $P$ is countably compact.

If $X=\{x_{n}:n\in \mathbb{N}\}$ is an infinite subset of $P$, it is clear that $x_{n}\in P \subset V_{n}\subset U_{n}$ for each $n\in \mathbb{N}$. Therefore $X$ has an accumulation point in $P$ and $P$ is countably compact.

{\bf Claim 3.} The family $\{V_{n}:n\in \mathbb{N}\}$ forms a base of neighborhoods of $P$ in $G$.

For an arbitrary open neighborhood $V$ of $P$ in $G$. If $V_{n}\setminus V\not =\emptyset$ for each $n\in \mathbb{N}$, fix points $x_{n}\in V_{n}\setminus V$. Then $X=\{x_{n}:n\in \mathbb{N}\}$ has no accumulation points in $G$, which contradicts with the Claim 1 above. Therefore, $\{V_{n}:n\in \mathbb{N}\}$ is an open neighborhood base of $P$ in $G$.

{\bf Claim 4.} The left coset space $G/P$ is metrizable.

Apply Lemma \ref{s} to choose a continuous prenorm $N$ on $G$ which satisfies $$N(gyr[x,y](z))=N(z)$$ for any $x, y, z\in G$ and $$\{x\in G: N(x)<1/2^{n}\}\subset V_{n}\subset \{x\in G: N(x)\leq 2/2^{n}\},$$ for each integer $n\geq 0$.
It is clear that $N(x)=0$ if and only if $x\in P$.

We claim that $N(x\oplus p)=N(x)$ for every $x\in G$ and $p\in P$. Indeed, for every $x\in G$ and $p\in P$, $N(x\oplus p)\leq N(x)+N(p)=N(x)+0=N(x)$.  Moreover, by the definition of $N$, we observe that $N(gyr[x, y](z))=N(z)$ for every $x, y, z\in G$. Since $H$ is a $L$-subgyrogroup, it follows from Lemma~\ref{a} that
\begin{eqnarray}
N(x)&=&N((x\oplus p)\oplus gyr[x,p](\ominus p))\nonumber\\
&\leq&N(x\oplus p)+N(gyr[x,p](\ominus p))\nonumber\\
&=&N(x\oplus p)+N(\ominus p)\nonumber\\
&=&N(x\oplus p).\nonumber
\end{eqnarray}
Therefore, $N(x\oplus p)=N(x)$ for every $x\in G$ and $p\in P$.

Now define a function $d$ from $G\times G$ to $\mathbb{R}$ by $d(x,y)=|N(x)-N(y)|$ for all $x,y\in G$. Obviously, $d$ is continuous. We show that $d$ is a pseudometric.

\smallskip
(1) For any $x, y\in G$, if $x=y$, then $d(x, y)=|N(x)-N(x)|=0$.

\smallskip
(2) For any $x, y\in G$, $d(y, x)=|N(y)-N(x)|=|N(x)-N(y)|=d(x, y)$.

\smallskip
(3) For any $x, y, z\in G$, we have
\begin{eqnarray}
d(x, y)&=&|N(x)-N(y)|\nonumber\\
&=&|N(x)-N(z)+N(z)-N(y)|\nonumber\\
&\leq&|N(x)-N(z)|+|N(z)-N(y)|\nonumber\\
&=&d(x, z)+d(z, y).\nonumber
\end{eqnarray}

If $x'\in x\oplus P$ and $y'\in y\oplus P$, there exist $p_{1},p_{2}\in P$ such that $x'=x\oplus p_{1}$ and $y'=y\oplus p_{2}$, then $$d(x', y')=|N(x\oplus p_{1})-N(y\oplus p_{2})|=|N(x)-N(y)|=d(x, y).$$ This enables us to define a function $\varrho $ on $G/P\times G/P$ by $$\varrho (\pi _{p}(x),\pi _{p}(y))=d(\ominus x\oplus y, 0)+d(\ominus y\oplus x, 0)$$ for any $x, y\in G$.

It is obvious that $\varrho $ is continuous, and we verify that $\varrho $ is a metric on $Y=G/P$.

\smallskip
(1) Obviously, for any $x, y\in G$, then
\begin{eqnarray}
\varrho (\pi _{P}(x),\pi _{P}(y))=0&\Leftrightarrow&d(\ominus x\oplus y, 0)=d(\ominus y\oplus x, 0)=0\nonumber\\
&\Leftrightarrow&N(\ominus x\oplus y)=N(\ominus y\oplus x)=0\nonumber\\
&\Leftrightarrow&\ominus x\oplus y\in P\ \mbox{and}\ \ominus y\oplus x\in P\nonumber\\
&\Leftrightarrow&y\in x+P\ \mbox{and}\ x\in y+P\nonumber\\
&\Leftrightarrow&\pi _{P}(x)=\pi _{P}(y).\nonumber
\end{eqnarray}

\smallskip
(2) For every $x,y\in G$, it is obvious that $\varrho (\pi _{P}(y), \pi _{P}(x))=\varrho (\pi _{P}(x),\pi _{P}(y))$.

\smallskip
(3) For every $x, y, z\in G$, it follows from \cite[Theorem 2.11]{UA2005} that
\begin{eqnarray}
\varrho (\pi _{P}(x),\pi _{P}(y))&=&N(\ominus x\oplus y)+N(\ominus y\oplus x)\nonumber\\
&=&N((\ominus x\oplus z)\oplus gyr[\ominus x,z](\ominus z\oplus y))\nonumber\\
&&+N((\ominus y\oplus z)\oplus gyr[\ominus y,z](\ominus z\oplus x))\nonumber\\
&\leq&N(\ominus x\oplus z)+N(gyr[\ominus x,z](\ominus z\oplus y))\nonumber\\
&&+N(\ominus y\oplus z)+N(gyr[\ominus y,z](\ominus z\oplus x))\nonumber\\
&=&N(\ominus x\oplus z)+N(\ominus z\oplus y)+N(\ominus y\oplus z)+N(\ominus z\oplus x)\nonumber\\
&=&d(\ominus x\oplus z, 0)+d(\ominus z\oplus x, 0)+d(\ominus z\oplus y, 0)+d(\ominus y\oplus z, 0)\nonumber\\
&=&\varrho (\pi _{P}(x),\pi _{P}(z))+\varrho (\pi _{P}(z),\pi _{P}(y)).\nonumber
\end{eqnarray}

Let us verify that $\varrho$ generates the quotient topology of the space $Y$. Given any points $x\in G$, $y\in Y$ and any $\varepsilon>0$, we define open balls, $$B(x, \varepsilon)=\{x'\in G: d(x',x)<\varepsilon\}$$ and $$B^{*}(y, \varepsilon)=\{y'\in G/P: \varrho (y',y)<\varepsilon\}$$ in $X$ and $Y$, respectively. Obviously, if $x\in G$ and $y=\pi _{P}(x)$, then we have $B(x, \varepsilon)=\pi ^{-1}_{P}(B^{*}(y, \varepsilon))$. Therefore, the topology generated by $\varrho$ on $Y$ is coarser than the quotient topology. Suppose that the preimage $O=\pi ^{-1}_{P}(W)$ is open in $G$, where $W$ is a non-empty subset of $Y$. For every $y\in W$, there exists $x\in G$ such that $\pi (x)=y$, then we have $\pi ^{-1}_{P}(y)=x\oplus P\subset O$. Since $\{V_{n}: n\in \omega\}$ is a base for $G$ at $P$, there exists $n\in \omega$ such that $x\oplus V_{n}\subset O$. Then there exists $\delta>0$ such that $B(x, \delta)\subset x\oplus V_{n}$. Therefore, we have $\pi ^{-1}_{P}(B^{*}(y, \delta))=B(x, \delta)\subset x\oplus V_{n}\subset O$. It follows that $B^{*}(y, \delta)\subset W$. So the set $W$ is the union of a family of open balls in $(Y, \varrho)$. Hence, $W$ is open in $(Y, \varrho)$, which proves that the metric and quotient topologies on $Y=G/P$ coincide. Therefore, the left coset space $G/P$ is metrizable.

{\bf Claim 5.} The canonical mapping $\pi :G\rightarrow G/P$ is closed.

Suppose that $F$ is a closed subset of $G$ and suppose further that $y\in Y\setminus \pi (F)$. Then, there exists a point $x\in G$ such that $\pi (x)=y$. Therefore, the coset $x\oplus P=\pi^{-1}(y)$ is disjoint from $F$, thus $O=G\setminus F$ is a neighborhood of $x\oplus P$ in $G$. Since $\{x\oplus V_{n}:n\in \mathbb{N}\}$ is a base of open neighborhood of $x\oplus P$ in $G$, we can find $n\in \mathbb{N}$ such that $x\oplus P\subset x\oplus V_{n}\subset O$. Therefore, $\pi (x\oplus V_{n+1})$ is an open neighborhood of $y$ disjoint from $\pi (F)$, so $\pi (F)$ is closed in $Y$. Hence, the canonical mapping $\pi$ is closed.
\end{proof}

\begin{corollary}\label{3tl1}
Every strongly countably complete (locally) pseudocompact strongly topological gyrogroup $G$ contains a closed countably compact $L$-subgyrogroup $H$ such that the quotient space $G/H$ is metrizable and (locally) compact, and the canonical mapping $\pi$ is closed.
\end{corollary}

\begin{proof}
Let $H$ be a closed $L$-subgyrogroup of $G$ as the $L$-subgyrogroup $P$ in Theorem \ref{3dl1}. Then the metrizable space $G/H$ is (locally) pseudocompact as an open continuous image of (locally) pseudocompact space $G$ by \cite[Theorem 3.7]{BL}. Moreover, every metrizable (locally) pseudocompact space is (locally) compact.
\end{proof}

\begin{lemma}\label{3yl2}
Let $G$ be a topological gyrogroup, and let $H$ be a closed  countably compact $L$-subgyrogroup of $G$. If $D$ is an infinite closed discrete subsets of $G$, $\pi (D)$ is infinite in the quotient space $G/H$.
\end{lemma}

\begin{proof}
Let $\mathcal{D}=\{\{d\}:d\in D\}$. Then $\mathcal{D}$ is a family of locally finite subsets of $G$. Since the $L$-subgyrogroup $x\oplus H$ is closed in $G$, we have that $\mathcal{D}|_{x\oplus H}$ is also locally finite in $x\oplus H$. It follows from the countable compactness of $H$ that $\mathcal{D}|_{x\oplus H}$ is finite. Hence, $D\cap (x\oplus H)$ is finite for all $x\in G$. Therefore, $\pi (D)$ is infinite.
\end{proof}

\begin{theorem}
Every strongly countably complete (locally) pseudocompact strongly topological gyrogroup $G$ is (locally) countably compact.
\end{theorem}

\begin{proof}
Assume that $G$ is pseudocompact. It follows from Corollary \ref{3tl1} that $G$ contains a closed countably compact $L$-subgyrogroup $H$ such that the quotient space $G/H$ is compact, and the canonical mapping $\pi$ is closed. Suppose on the contrary that $G$ is not countably compact, then there exists an infinite closed discrete subset $D$ of $G$. By Lemma \ref{3yl2}, $\pi (D)$ is infinite. Moreover, since $\pi$ is a closed mapping, $\pi (D)$ is closed and discrete in $G/H$. However, $G/H$ is compact, which is a contradiction. Then $G$ is countably compact.

If $G$ is locally pseudocompact, By Corollary \ref{3tl1}, there exists a compact neighborhood $V$ of $\pi (0)$ in the quotient space $G/H$. Since $\pi$ is a closed mapping, $\pi^{-1}(V)$ is a countably compact neighborhood of $0$ in $G$. Therefore, $G$ is locally countably compact.
\end{proof}

Until now, we do not know whether the inverse of Theorem \ref{3dl1} also holds. Therefore, we pose the following question.

\begin{question}
Let $G$ be a strongly topological gyrogroup with a symmetric neighborhood base $\mathscr U$. If $G$ contains a closed, countably compact, admissible subgyrogroup $P$ such that the quotient space $G/P$ is metrizable and the canonical homomorphism $\pi :G\rightarrow G/P$ is closed, is $G$ strongly countably complete?
\end{question}

{\bf Acknowledgements}. The authors are thankful to the
anonymous referees for valuable remarks and corrections and all other sort of help related to the content of this article.

\end{document}